\documentclass{article}
\usepackage[T1]{fontenc}
\usepackage[utf8]{inputenc}
\usepackage{geometry}
\usepackage{amsmath,amsfonts,amssymb,amsthm,thmtools}
\usepackage{enumerate}

\usepackage[dvipsnames]{xcolor} 
\usepackage[color=Orange!50!white, textwidth=22mm]{todonotes}
\usepackage[unicode=true,
 bookmarks=false,
 breaklinks=false,pdfborder={0 0 1},colorlinks=false]
 {hyperref}
\usepackage[capitalize]{cleveref}

\usepackage{comment}
\makeatletter
\theoremstyle{plain}
\newtheorem{thm}{\protect\theoremname}
\theoremstyle{plain}
\newtheorem{lemma}[thm]{\protect\lemmaname}

\theoremstyle{remark}

\theoremstyle{plain}

\theoremstyle{plain}

\theoremstyle{plain}

\theoremstyle{definition}
\newtheorem{remark}{Remark}

\makeatother

\providecommand{\claimname}{Claim}
\providecommand{\corollaryname}{Corollary}
\providecommand{\lemmaname}{Lemma}
\providecommand{\theoremname}{Theorem}
\providecommand{\problemname}{Problem}
\providecommand{\propositionname}{Proposition}

\newcommand{\FF}{\mathbb F}
\newcommand{\ZZ}{\mathbb Z}

\usepackage{pifont}

\begin{document}

\title{A new lower bound for two-color van der Waerden numbers}
\author{Marcelo Campos\thanks{IMPA, Estrada Dona Castorina 110, Rio de Janeiro, 22460-320. Email: {\tt marcelo.campos@impa.br}\\ Research supported by Serrapilheira (grant R-2412-51283)}\and Jacob Fox\thanks{Department of Mathematics, Stanford University, Stanford, CA 94305. Email: {\tt jacobfox@stanford.edu}.\\ Research supported by NSF awards DMS-2452737 and DMS-2154129.}\and Carl Schildkraut\thanks{Department of Mathematics, Stanford University, Stanford, CA 94305. Email: {\tt carlsch@stanford.edu} \\Research supported by an NSF Graduate Research Fellowship Program under Grant No. DGE-2146755.}}
\date{}
\maketitle

\begin{abstract}
The \emph{van der Waerden number} $w(k)$ is the smallest positive integer $N$ such that every two-coloring of $\{1,2,\ldots,N\}$ contains a monochromatic $k$-term arithmetic progression. 
We prove that $w(k) \geq (1-o(1))k2^{k-1}$ holds for all positive integers $k$. 
This verifies a conjecture of Erd\H{o}s. 
In 1968, Berlekamp proved the same result when $k-1$ is prime. The coloring for general $k$ can be viewed as a product of Berlekamp's colorings for various primes. It was found by ChatGPT 5.6 Sol Pro.

\end{abstract}

\section{Introduction}

The van der Waerden number $w(k)$ is the smallest positive integer $N$ such that every two-coloring of the first $N$ positive integers contains a monochromatic $k$-term arithmetic progression. Van der Waerden \cite{vanderWaerden} proved that these numbers exist. 

The first exponential lower bound on $w(k)$ was proved by Erd\H{o}s and Rado \cite{ErRa52} in 1952 in an early application of the probabilistic method. Berlekamp \cite{Berlekamp} in 1968 used an algebraic construction to prove that $w(p+1) > p2^p$ for every prime $p$. 

Erd\H{o}s asked in many different papers over the years to improve the bounds on $w(k)$. In particular, Erd\H{o}s conjectured \cite{Erdos1980} that $\lim_{k \to \infty}w(k)/2^k = \infty$. He also offered \$500 to prove or disprove that van der Waerden numbers have superexponential growth, that is, $\lim_{k \to \infty}w(k)^{1/k}=\infty$. Fox and Hunter \cite{FoxHunter} proved that three-color van der Waerden numbers have superexponential growth.

With the development of more sophisticated applications of the probabilistic method, further lower bound improvements on $w(k)$ were successively made  \cite{Schmidt,erdos lovasz,Spencer77,szabo,Kozik,KuSh,KS16}, with the previous best known general lower bound by Kozik and Shabanov \cite{KS16} in 2016 of the form $w(k) \geq c2^k$ for some constant $c>0$ and all $k$.  

We extend Berlekamp's theorem to all positive integers $k$. This verifies Erd\H{o}s' conjecture. 

\begin{thm}\label{main}
    For every positive integer $k$ we have $w(k)\geq (1-o(1))k2^{k-1}$.
\end{thm}

\noindent \textbf{AI Use.} 
The coloring used to prove \cref{main}, as well as a proof of \cref{main}, were generated by a query to ChatGPT 5.6 Sol Pro.
The proof we present, while drawing somewhat on the proof given by ChatGPT, was produced by the authors. Specifically, our proof is a modification of an exposition we wrote of Berlekamp's proof in the $k-1$ prime case. All of the writing is our own.

\noindent \textbf{Acknowledgments.} We would like to thank Zach Hunter for helpful comments. 

\section{The proof}

We first generalize the setting slightly.
In an abelian group, a $k$-term arithmetic progression ($k$-AP) is a sequence $a,a+d,\ldots,a+(k-1)d$ with $d$ nonzero. (Note that we do not require that the elements of a $k$-AP are all distinct.)
Define the \emph{cyclic van der Waerden number} $w_{\mathrm{cyc}}(k)$ as the minimum $n$ such that, for every $N \geq n$, every coloring $c\colon\ZZ/N\ZZ \rightarrow \{0,1\}$ contains a monochromatic $k$-AP. 

The following simple lemma bounds van der Waerden numbers in terms of their cyclic variants.

\begin{lemma}\label{fromcyclic}
We have $w(k) \geq (k-1)(w_{\mathrm{cyc}}(k)-1)+1$ for every $k \geq 2$. 
\end{lemma}
\begin{proof}
    Let $n=w_{\mathrm{cyc}}(k)-1$. 
    By the definition of $w_{\mathrm{cyc}}$, there exists a coloring $c\colon\ZZ/n\ZZ \rightarrow \{0,1\}$ with no monochromatic $k$-AP.
    Let $N=(k-1)n$ and consider the periodic coloring $c'\colon\{1,2,\ldots,N\} \rightarrow \{0,1\}$ with period $n$ given by $c'(x)=c(x\bmod n)$.
    By our choice of $c$, any monochromatic $k$-AP in coloring $c'$ must have common difference a multiple of $n$.
    However, $\{1,2,\ldots,N\}$ contains no $k$-AP with common difference a multiple of $n$ as any $k$-AP with terms in $\{1,\ldots,(k-1)n\}$ has common difference at most $n-1$. 
    We conclude that $c'$ admits no monochromatic $k$-AP, and hence $w(k)\geq N+1$.
\end{proof}

Berlekamp's lower bound on $w(p+1)$ nearly\footnote{See the discussion in Remark 2.} follows from \cref{fromcyclic} and the following result.

\begin{thm}[{Berlekamp \cite{Berlekamp}}]\label{berlekamp-cyc}
    For prime $p$, we have $w_{\mathrm{cyc}}(p+1)>2^p-1$.
\end{thm}

We will prove the following generalization of \cref{berlekamp-cyc}.

\begin{lemma}\label{gencyclic}
    Let $p_1,\ldots,p_t$ be distinct primes. Then
    \[w_{\mathrm{cyc}}(p_1+\cdots+p_t+1)>(2^{p_1}-1)\cdots(2^{p_t}-1).\]
\end{lemma}

Let $p$ be any prime number, and consider a finite field $\FF_{2^p}$ of order $2^p$. We record three facts about $\FF_{2^p}$ and set some notation:
\begin{enumerate}[(a)]
    \item The multiplicative group of units of $\FF_{2^p}$ is a cyclic group of order $2^p-1$. 
    Let $\alpha_p$ be an arbitrary generator of this cyclic group.

    \item The field $\FF_{2^p}$ forms a vector space over $\FF_2$ of dimension $p$. Let $U_p$ be an arbitrary codimension-$1$ subspace of this vector space.
    
    \item \label{fact:degp} Since $p$ is prime, the prime field $\FF_2$ is the only proper subfield of $\FF_{2^p}$. In particular, every element of $\FF_{2^p}\setminus\FF_2$ generates $\FF_{2^p}$ as a field extension of $\FF_2$, and so no such element satisfies a nonzero polynomial equation over $\mathbb{F}_2$ of degree strictly less than $p$.
\end{enumerate}
The construction witnessing \cref{berlekamp-cyc} is the coloring $c_p\colon \ZZ/(2^p-1)\ZZ\to\ZZ/2\ZZ$ defined by\footnote{In a certain sense, the coloring of $\ZZ/(2^p-1)\ZZ$ obtained in this way does not depend on the choice of $\alpha_p$ and the subspace $U_p$. All such colorings can be transformed into each other by pre-composition by an invertible linear map $x\mapsto ax+b$ on the ring $\ZZ/(2^p-1)\ZZ$. (Changing $U_p$ corresponds to the translation maps $x\mapsto x+b$, and changing $\alpha_p$ corresponds to multiplication.)}
\[c_p(n)=\begin{cases}0&\text{if }\alpha_p^n\in U_p\\1&\text{otherwise.}\end{cases}\]
To motivate our approach, we outline a proof of \cref{berlekamp-cyc}. 
Any monochromatic (non-trivial) arithmetic progression of color $0$ will correspond to a (non-trivial) geometric progression in $\FF_{2^p}$ which is contained in the codimension 1 subspace $U_p$. For a monochromatic arithmetic progression of color $1$, we subtract consecutive terms of the corresponding geometric progression to obtain a new geometric progression which is also contained in the subspace $U_p$. 
We can then use fact~\eqref{fact:degp} to say that any geometric progression with length at least $p$ must span $\FF_{2^p}$, contradicting  its containment in the proper subspace $U_p$. 

The construction we use to prove \cref{gencyclic} comes from combining $c_p$ for various $p$.
Let $\mathcal P=\{p_1,\ldots,p_t\}$ be a set of distinct primes, and set
\[s=\sum_{p\in\mathcal P}p;\qquad Q=\prod_{p\in\mathcal P}(2^p-1).\]
Consider the coloring $c_{\mathcal P}\colon\ZZ/Q\ZZ\to\ZZ/2\ZZ$ defined by
\[c_{\mathcal P}(n)=\sum_{p\in\mathcal P}c_p(n\bmod 2^p-1).\]
In the proof of Lemma~\ref{gencyclic}, we need the following generalization of fact~\eqref{fact:degp}.
\begin{lemma}\label{lem:degp}
    Let $I$ be any finite set of primes, and take $\gamma_p\in\FF_{2^p}\setminus\FF_2$ arbitrary for each $p\in I$.
    Let $v_j:=\bigoplus_{p\in I}\gamma_p^j(\gamma_p-1)$.
    If $s=\sum_{p\in I}p$, then the elements $v_0,\ldots,v_{s-1}$ span $\bigoplus_{p\in I}\FF_{2^p}$.
\end{lemma}
\begin{proof}
    Since $\dim\bigoplus_{p\in I}\FF_{2^p}=s$, it suffices to prove that $v_0,\ldots,v_{s-1}$ are linearly independent.
    If they were linearly dependent, then there would exist some nontrivial polynomial $f$ over $\FF_2$ of degree at most $s-1$ for which
    \[\bigoplus_{p\in I}f(\gamma_p)(\gamma_p-1)=0.\]
    Since $\gamma_p\not\in\FF_2$, this implies that $f(\gamma_p)=0$ for each $p$. So, for each $p\in I$, such an $f$ must be a multiple of the minimal polynomial of $\gamma_p$. By fact~\eqref{fact:degp}, the minimal polynomial of $\gamma_p$ has degree $p$. These minimal polynomials, which are each irreducible, are thus all relatively prime. Hence, $f$ is a multiple of the product of these minimal polynomials,  and so $\deg f\geq\sum_{p\in I}p=s$, a contradiction.
\end{proof}

\begin{proof}[Proof of \cref{gencyclic}]
    We shall show that $c_{\mathcal P}\colon \ZZ/Q\ZZ\to\ZZ/2\ZZ$ avoids monochromatic~$(s+1)$-APs. 
    Suppose to the contrary that there is such a monochromatic progression $a,a+d,\ldots,a+sd$ with $d\in\ZZ/Q\ZZ\setminus\{0\}$. 
    Write $\beta_p=\alpha_p^a$ and $\gamma_p=\alpha_p^d$.
    
    Let $I=\{p\in\mathcal P:\gamma_p\neq 1\}$.
    We first see that $I$ is nonempty. Indeed, if $\alpha_p^d=\gamma_p=1$ for all $p\in\mathcal P$, then $2^p-1\mid d$ for every $p\in\mathcal P$. Since $\mathcal P$ consists of distinct primes, the values $\{2^p-1:p\in\mathcal P\}$ are relatively prime. 
    So, this would imply $Q\mid d$, which contradicts the assumption that $d\in\ZZ/Q\ZZ\setminus\{0\}$.

    Consider the $\FF_2$-vector space $V:=\bigoplus_{p\in\mathcal P}\FF_{2^p}$ along with the component projections $\pi_p\colon V\to\FF_{2^p}$. Define
    \[U:=\{v\in V:\pi_p(v)\not\in \beta_p^{-1}U_p\text{ for an even number of }p\in\mathcal P\}.\]
    The set $U$ is a subspace of $V$ of codimension $1$.
    Moreover, the element
    \[u_j:=\bigoplus_{p\in\mathcal P}\gamma_p^j\]
    satisfies
    \[c_{\mathcal P}(a+jd)=\sum_{p\in\mathcal P}c_p(a+jd)=\sum_{p\in\mathcal P}\begin{cases}0&\text{if }\beta_p\gamma_p^j\in U_p\\1&\text{otherwise}\end{cases}=\sum_{p\in\mathcal P}\begin{cases}0&\text{if }\gamma_p^j\in \beta_p^{-1}U_p\\1&\text{otherwise}\end{cases}=\begin{cases}0&\text{if }u_j\in U\\1&\text{otherwise.}\end{cases}\]
    In other words, the assumption that $a,a+d,\ldots,a+sd$ is monochromatic is equivalent to the statement that the elements $u_0,\ldots,u_s\in V$ either all lie in $U$ or all lie outside of $U$. 
    In either case, we have
    \begin{equation}\label{eq:vj}
    v_j:=\bigoplus_{p\in\mathcal P}\gamma_p^j(\gamma_p-1)=u_{j+1}-u_j\in U
    \end{equation}
    for each $0\leq j<s$.
    
    Let $W=\{v \in V:\pi_p(v)=0~\textrm{for all}~p \in \mathcal{P} \setminus I\}$, so $v_j\in W$ for each $j$. 
    The space $W$ contains the space $0\oplus\cdots\oplus 0\oplus \FF_{2^p}\oplus 0\oplus\cdots\oplus 0$ for any $p\in I$, but this space is not contained in $U$. Hence, $W\not \subset U$. 
    However, as $\sum_{p \in I} p \leq s$, Lemma~\ref{lem:degp} implies that $\{v_0,\ldots,v_{s-1}\}$ spans $W$, contradicting \eqref{eq:vj}.
\end{proof}

We use the following lemma to finish the proof of Theorem~\ref{main}.

\begin{lemma}\label{primesums}
    Every sufficiently large positive $N$ can be expressed as a sum of either three or four distinct primes, each of which is at least $N/4-o(N)$.
\end{lemma}
\begin{proof}
    It follows from work of Matom\"aki, Maynard, and Shao \cite{MMS} that every sufficiently large odd integer $N$ is the sum of three distinct odd primes, each at least $N/3 - N^{11/20+o(1)}$.\footnote{The main result \cite[Theorem~1.1]{MMS} is this statement, but without the condition that the primes are distinct. To recover a representation with distinct primes from their proof, one need only choose the residue classes $b_1,b_2,b_3$ in \cite[Proof of Theorem~2.1 from Proposition~3.1]{MMS} to be distinct.}
    As there is always a prime in $[x,x+x^{0.525}]$ for $x$ sufficiently large \cite{BHP}, it similarly follows that every sufficiently large even integer $N$ is the sum of four distinct odd primes, each at least $N/4 - N^{11/20+o(1)}$. Indeed, we can pick $x$ of the form $N/4+N^{11/20+o(1)}$ so that, after picking a prime $p \in [x,x+x^{0.525}]$, we can similarly write $N-p$ as the sum of three odd primes each less than $x \leq p$. 
\end{proof}

\begin{proof}[Proof of \cref{main}]
    Let $k$ be a sufficiently large positive integer. Using \cref{primesums}, write $k=p_1+\cdots+p_t+1$ with $t\in\{3,4\}$ for distinct primes $p_1,\ldots,p_t$ each at least $k/4-o(k)$. \cref{gencyclic} implies that
    \[w_{\mathrm{cyc}}(k)>(2^{p_1}-1)\cdots(2^{p_t}-1)=2^{k-1}\prod_{i=1}^t(1-2^{-p_i})\geq2^{k-1}\left(1-2^{-k/4+o(k)}\right).\]
    By \cref{fromcyclic}, we conclude
    \[w(k)\geq (k-1)(w_{\mathrm{cyc}}(k)-1)+1\geq(1-o(1))k2^{k-1},\]
    where the $o(1)$ term can be taken to be $k^{-1}+2^{-k/4+o(k)}$.
\end{proof}

We end with a few remarks on alternatives and variations of the above arguments.

\begin{remark}
    A weaker version of \cref{primesums} can be proven using only Bertrand's postulate, which enables the recovery of \cref{main} up to a constant factor. Indeed, for each integer $k>2$, one can write either $k-1$ or $k-2$ as a sum of distinct primes: beginning with $n=k-1$, successively subtract off a prime from $n$ in the interval $(n/2,n]$ until $n=0$ or $n=1$ remains. Using such a representation $p_1+\cdots+p_t\in\{k-1,k-2\}$, \cref{fromcyclic,gencyclic} give
    \[w(k)\geq (k-1)(w_{\mathrm{cyc}}(p_1+\cdots+p_t+1)-1)+1>(k-1)2^{k-2}\prod_{i=1}^t(1-2^{-p_i})>\frac1{10}k2^k.\]
\end{remark}

\begin{remark}
    Berlekamp observed that the coloring he gave for Theorem~\ref{berlekamp-cyc} with $s=p$ has no $s$-term arithmetic progression in color $0$ and no $(s+1)$-term arithmetic progression in color $1$. This also holds for the coloring we get in Lemma~\ref{gencyclic} with $s=p_1+\cdots+p_t$. Using this, Berlekamp was able to extend the interval coloring by adding $s$ additional elements (roughly $s/2$ on either side) to get a coloring of an interval of length $s$ longer with no monochromatic $(s+1)$-term arithmetic progression than one gets from directly applying Lemma~\ref{fromcyclic}. Letting $k=s+1$, one may show that the same modification works to improve our bound on $w(k)$ by an additive $k-1$. We have chosen to omit showing this for brevity, and because it does not improve the form of our bound in \cref{main}.
\end{remark}

\begin{remark}
    Burkert and Johnson \cite{BuJo} introduced the following alternative notion of cyclic van der Waerden numbers: let $w_c(k)$ be the minimum $N_0$ such that for every integer $N \geq N_0$, every two-coloring of $\mathbb{Z}/N\mathbb{Z}$ contains a monochromatic $k$-term arithmetic progression {\it with distinct terms}. It is easy to check that the proof of \cref{fromcyclic} extends to show 
    $$w(k) \geq w_c(k) \geq (k-1)(w_{\textrm{cyc}}(k)-1)+1.$$
    Consequently, Theorem \ref{main} holds with $w_c(k)$ in place of $w(k)$.
\end{remark}

\end{document}